\documentclass[11pt]{article}

\usepackage{graphicx}
\usepackage{pdfpages}
\usepackage{amssymb}
\usepackage{amsmath}
\usepackage{amsthm}
\usepackage{latexsym} 
\usepackage{color}
\usepackage{pstricks} 
\usepackage{psfrag}
\usepackage{rotating} 
\usepackage{amsmath,amsthm,amssymb,graphicx,bbm,cancel,color,mathrsfs,todonotes, xypic,hyperref, pinlabel}

\newtheorem{theorem}{Theorem}
\newtheorem{lemma}[theorem]{Lemma}
\newtheorem{corollary}[theorem]{Corollary}
\newtheorem{conjecture}[theorem]{Conjecture}
\newtheorem{proposition}[theorem]{Proposition}

\newtheorem{observation}[theorem]{Observation} 

\theoremstyle{definition}
\newtheorem{definition}[theorem]{Definition}
\newtheorem{example}[theorem]{Example}

\usepackage{hyperref}

\newcommand\vn{\varnothing}

\newcommand\R{\mathbb{R}}

\newcommand\Z{\mathbb{Z}}
\newcommand\CAT{\mathrm{CAT}}

\newcommand\Isom{\text{Isom}}
 
\newcommand\lb{\lbrace} 
\newcommand\rb{\rbrace} 
\newcommand{\ca}{\mathcal}
\newcommand{\ov}{\overline}
\newcommand{\ti}{\times}

\newcommand{\Ga}{\Gamma}
\newcommand{\om}{\omega}
\newcommand{\Lam}{\Lambda}

\newcommand{\dmo}{\DeclareMathOperator}
\dmo{\pa}{\partial}
\dmo{\PD}{PD}
\dmo{\diam}{diam}
\dmo{\intt}{int}
\dmo{\Sh}{Sh}
\newcommand{\bb}{\mathbb}\newcommand{\si}{\sigma}
\newcommand{\de}{\delta}
\newcommand{\ga}{\gamma}
\newcommand{\Om}{\Omega}
\newcommand{\al}{\alpha}
\newcommand{\be}{\beta}
\newcommand{\la}{\lambda}
\newcommand{\ld}{\ldots}
\newcommand{\cd}{\cdots}
\newcommand{\sm}{\setminus}
\newcommand{\ra}{\rightarrow}
\newcommand{\hra}{\hookrightarrow}
\newcommand{\sbs}{\subset}
\newcommand{\scr}{\mathscr}
\newcommand{\rest}[2]{#1\bigr\vert_{#2}}

\title{On groups with $S^2$ Bowditch boundary} 

\begin{document} 
\author{Bena Tshishiku}
\author{Bena Tshishiku and Genevieve Walsh}


\date{\today}


\maketitle

\begin{abstract} 
We prove that a relatively hyperbolic pair $(G,\ca P)$ has Bowditch boundary a 2-sphere if and only if it is a 3-dimensional Poincar\'e duality pair. We prove this by studying the relationship between the Bowditch and Dahmani boundaries of relatively hyperbolic groups. 
\end{abstract}

\section{Duality for groups with Bowditch boundary $S^2$}

The goal of this paper is to study the duality properties of relatively hyperbolic pairs $(G,\ca P)$. This builds on work of Bestvina--Mess \cite{bestvina-mess}, who show that the duality properties of a hyperbolic group $G$ are encoded in its Gromov boundary $\pa G$; for example, a hyperbolic group $G$ with Gromov boundary $\pa G\simeq S^{n-1}$ is a $\PD(n)$ group. By analogy, one might hope for a similar result for relatively hyperbolic pairs $(G,\ca P)$ with the Gromov boundary replaced by the Bowditch boundary $\pa_B(G,\ca P)$. This would follow immediately from \cite{bestvina-mess} if the Bowditch boundary gave a Z-set compactification of $G$, but unfortunately this is not the case, and \cite{bestvina-mess} does not imply that  $(G, \mathcal{P})$ is a duality pair whenever $\pa_B(G, \mathcal{P})\simeq S^{n-1}$.  Instead we work with the \emph{Dahmani boundary} $\pa_D(G, \mathcal{P})$ (see \S\ref{s:Dahmani}), which does give a Z-set compactification. Our main theorem determines the Dahmani boundary when $\pa_B(G,\ca P)\simeq S^2$. 

\begin{theorem}\label{thm:sierpinski}
A relatively hyperbolic group $(G, \mathcal{P})$ with Bowditch boundary $\pa_B(G, \mathcal{P})\simeq S^2$ has Dahmani boundary $\pa_D(G, \mathcal{P})\simeq \scr S$ a Sierpinski carpet. 
\end{theorem}

As a corollary, we find that if $\pa_B(G,\ca P)$ is a 2-sphere then the same is true for the Dahmani boundary of the \emph{double of $G$ along $\ca P$} (see \S\ref{sec:double} for the definition). 

\begin{corollary}\label{cor:sphere}
Let $(G, \mathcal{P})$ be a relatively hyperbolic pair, and let $G_\de$ denote the double of $G$ along $\ca P$. If $\pa_B(G, \mathcal{P})\simeq S^2$, then $\pa_D(G_\de,\mathcal{P})\simeq S^2$. 
\end{corollary}

From Corollary \ref{cor:sphere} we obtain the following corollary, which is our main application.  A finitely presented group $G$ is an \emph{oriented Poincar\'e duality group of dimension $n$} (a $\PD(n)$ group) if for each $G$-module $A$ there are isomorphisms $H^i(G;A)\ra H_{n-i}(G;A)$ for each $i$, induced by cap product with a generator of $H_n(G;\Z)$. A relative version of this definition was introduced by Bieri--Eckmann \cite{bieri-eckmann-pdpairs}. We will only need a special case: a group pair $(G,\ca P)$ is a \emph{$\PD(3)$ pair} if each $P\in\ca P$ is the fundamental group of a closed surface and the double of $G$ along $\ca P$ is a $\PD(3)$ group; c.f.\ \cite[Cor.\ 8.5]{bieri-eckmann-pdpairs}.

\begin{corollary}\label{cor:duality}
Let $(G, \mathcal{P})$ be a  torsion-free relatively hyperbolic pair with Bowditch boundary $\pa_B(G, \mathcal{P})\simeq S^2$. Then $(G, \mathcal{P})$ is a $\PD(3)$ pair. 
\end{corollary}

The converse is also true.

\begin{theorem} \label{th:converse} Let $(G, \mathcal{P})$ be a  relatively hyperbolic pair.  If $(G, \mathcal{P})$ is a $\PD(3)$ pair, then   $\pa_B(G, \mathcal{P})\simeq S^2$. \end{theorem}

{\it Remark.} As a motivating example of Corollary \ref{cor:duality}, suppose $G$ is the fundamental group of a hyperbolic 3-manifold $M$ with $k$ cusps and $\ell$ totally geodesic  boundary components. Then $(G, \mathcal{P})$ is a relatively hyperbolic, where $\mathcal{P}$ consists of conjugates of the boundary and cusp subgroups $\lbrace P_1,\ld,P_{k+\ell} \rbrace$. On the one hand, $(G,\ca P)$ is a PD(3) pair because $M$ is a $K(G,1)$ and manifolds satisfy Poincar\'e duality. (Alternatively, remove neighborhoods of the cusps and take the double.) On the other hand, Corollary \ref{cor:duality} gives a geometric-group-theoretic proof since $\pa_B(G,\ca P)\simeq S^2$ (see e.g.\ \cite{ruane,tran}).  A different, homological proof of Corollary \ref{cor:sphere} and Theorem \ref{th:converse} is given in Manning-Wang  \cite[Corollary 4.3]{manning-wang}.

{\it Relation to the Wall and Cannon conjectures.} The Wall conjecture \cite{wall_problems}  posits (in dimension 3) that any $\PD(3)$ group is the fundamental group of a closed aspherical 3-manifold. Similarly, one would conjecture that if $(G,\ca P)$ is a $\PD(3)$ pair, then $G$ is the fundamental group of an aspherical 3-manifold with boundary, where $\ca P$ is the collection of conjugacy classes of the boundary subgroups. 

\begin{conjecture} [The relative Cannon conjecture] Let $(G, \mathcal{P})$ be a  relatively hyperbolic group pair with $G$ torsion-free. If $\pa_B(G, \mathcal{P})\simeq S^2$, then $G$ is the fundamental group of a finite volume hyperbolic $3$-manifold $M$.  Furthermore, the peripheral groups are the fundamental groups of the cusps and totally geodesic boundary components of $M$. 
\end{conjecture}

\begin{theorem} \label{thm:wall}If the Wall conjecture is true, then the relative Cannon conjecture is true.  
\end{theorem} 

Compare with \cite{kk}, which is similar. A slightly different theorem that the Cannon conjecture implies the relative Cannon conjecture when the peripheral subgroups are $\Z^2$, is given in \cite{GMS} with a completely different proof. 
Our version follows from Theorem \ref{thm:sierpinski}, Corollary \ref{cor:sphere}, and a result of Kapovich and Kleiner on the uniqueness of peripheral structures \cite[Theorem 1.5]{kk_duality}. Martin and Skora \cite{martin-skora} conjecture that convergence groups can be realized as Kleinian groups, which encompasses the Cannon and relative Cannon conjectures. 

We now explain the rough outline for Theorem \ref{thm:sierpinski}. 
\begin{enumerate}
\item In general there is a continuous surjection $c:\pa_D(G,\ca P)\ra\pa_B(G,\ca P)$. We collect some facts about the topology on $\pa_D(G,\ca P)$ and this map in Proposition \ref{prop:quotient}. In the case of Theorem \ref{thm:sierpinski}, we have a map $c:\pa_D(G,\ca P)\ra S^2$ such that $c^{-1}(z)$ is either a single point or a circle for each $z\in S^2$.
\item In Lemma \ref{lem:sierpinski} we identify conditions on a map $X\ra S^2$ that are sufficient to conclude that $X$ is a Sierpinski carpet. This gives a characterization of the Sierpinski carpet, which may be of independent interest.
\item We verify that the conditions of Lemma \ref{lem:sierpinski} are satisfied for the map $c:\pa_D(G,\ca P)\ra S^2$. One of the difficult parts is to show that if $\pa_B(G,\ca P)\simeq S^2$, then if $\pa_D(G,\ca P)_P$ is the quotient of $\pa_D(G,\ca P)$ obtained by collapsing all but one of the peripheral circles to points, then $\pa_D(G,\ca P)_P$ is homeomorphic to the closed disk. 
\end{enumerate}

{\it Remark.} It would be interesting to know a version of Theorem \ref{thm:sierpinski} when $\pa_B(G,\ca P)\simeq S^{n-1}$ for $n>3$, i.e.\ that in this case $\pa_D(G,\ca P)$ is an $(n-2)$-dimensional Sierpinski carpet. The methods of this paper show that this is true if one knows that each $P\in\ca P$ admits a $\ca Z$-boundary $\pa P\simeq S^{n-2}$. When $n=3$ this is automatic because the peripheral subgroups are always surface groups. 

{\it Section outline.} In \S\ref{s:Dahmani} we collect some facts about the Bowditch and Dahmani boundaries of a relatively hyperbolic group and their relation. In \S\ref{S:proofof2} we prove Theorem \ref{thm:sierpinski}, and in \S\ref{sec:cor} we prove Corollaries \ref{cor:sphere} and \ref{cor:duality} and Theorems \ref{th:converse} and \ref{thm:wall}.

{\it Acknowledgements.}  The authors thank Francois Dahmani for helpful conversations about his work, and thank Craig Guilbault for help with $\ca Z$-structures. The authors are grateful to the referee for carefully reading the paper, catching errors, and offering helpful observations that significantly improved the paper. The authors also acknowledge support from NSF grants DMS 1502794 and DMS 1709964, respectively. 

\section{Relatively hyperbolic groups and their boundaries} \label{s:Dahmani}

We will assume throughout that $G$ is finitely generated and thus so are the peripheral subgroups \cite{osin}. 

There are different notions of boundary for a relatively hyperbolic group. The most general definition is due to Bowditch \cite{bowditch}. Another boundary was defined by Dahmani \cite{dahmani-boundary} in the case when each peripheral subgroup \emph{admits a boundary}, i.e.\ for each $P\in\ca P$ there is a space $\pa P$ so that $P\cup\pa P$ is compact, metrizable, and $P\sbs P\cup\pa P$ is dense. In this section we describe the Bowditch and Dahmani boundaries. Our description of Dahmani's boundary differs slightly from that in \cite{dahmani-boundary} because we use the coned-off Cayley graph instead of the collapsed Cayley graph. This is required in order to allow $\ca P$ to contain more than one conjugacy class, as discussed in \cite[\S6]{dahmani-boundary}. Everything in this section will be done for general relatively hyperbolic groups, although the case with one conjugacy class of peripheral subgroups is the most rigorous case in \cite{dahmani-boundary}. In the next section we will specialize to the case $\pa_B(G,\ca P)\simeq S^2$.

\subsection{Relatively hyperbolic groups and the Bowditch boundary}\label{sec:bowditch-boundary}

Below $G$ is a group and $\ca P$ is a collection of subgroups of $G$ that consists of a finite number of conjugacy classes of $G$. Some authors use $\ca P$ to refer to a collection of conjugacy representatives, but we do not use this convention. This causes a minor notational conflict since the notion of PD($n$) pair (as discussed in the previous section) is reserved for a group with respect to a \emph{finite} collection of subgroups. However, this should not cause any confusion as we can pick any set $\{P_1,\ldots,P_d\}$ of conjugacy class representatives to be this finite collection. 

For a subgroup $P<G$ and $a\in G$, we denote ${}^aP:=aPa^{-1}$ for the (left) action of $G$ by conjugation.

{\it The coned--off Cayley graph.} Fix a relatively hyperbolic group $(G,\ca P)$, and let $P_1,\ld, P_d$ be representatives for the conjugacy classes in $\mathcal{P}$. Let $S$ be a generating set for $G$ that contains generating sets $S_i$ for each $P_i$.  Then the Cayley graph $\Ga(G)=\Ga(G,S)$ naturally contains the Cayley graphs $\Ga(P_i,S_i)$ for each $i=1,\ld,d$. If $P\in\ca P$ and $P=aP_ia^{-1}$, then we denote by $\Ga(P)\sbs\Ga(G)$ the subgraph $a\>\Ga(P_i,S_i)$; note that $\Ga(P)$ is isomorphic to a Cayley graph for $P$ since $\Ga({}^aP_i,{}^aS_i)\simeq\Ga(P_i,S_i)\simeq a\>\Ga(P_i,S_i)$. We form the {\it coned off Cayley graph} $\hat\Ga=\hat \Gamma(G, \mathcal{P}, S)$ by adding a vertex $*_{P}$ for each $P\in\ca P$ and adding edges of length $1/2$ from $*_{P}$ to each vertex of $\Ga(P)\sbs\Ga(G)$.

An oriented path $\gamma$ in $\hat\Ga$ is said to {\it penetrate} $P\in\ca P$ if it passes through the cone point $*_{P}$; its \emph{entering} and \emph{exiting} vertices are the vertices immediately before and after $*_{P}$ on $\gamma$. The path is {\it without backtracking} if once it penetrates $P\in\ca P$, it does not penetrate $P$ again. 

\begin{definition} The triple $(G, \mathcal{P} ,S)$ is said to have {\it bounded coset penetration} if for each $\lambda\ge1$, there is a constant $a=a(\lambda)$ such that if $\gamma$ and $\gamma'$ are $(\lambda,0)$ quasi-geodesics  without backtracking in $\hat \Gamma$ and with the same endpoints, then 
\begin{enumerate} 
\item[(i)] if $\gamma$ penetrates some $P\in\ca P$, but $\gamma'$ does not, then the distance between the entering and exiting vertices of $\gamma$ in $\Ga(P)$ is at most $a$; and 
\item[(ii)] if $\gamma$ and $\gamma'$ both penetrate $P$, then the distance between the entering vertices of $\gamma$ and $\gamma'$ in $\Ga(P)$ is at most $a$, and similarly for the exiting vertices. 
\end{enumerate}  
\end{definition} 

{\it Relative hyperbolicity and Bowditch boundary.} The pair $(G, \mathcal{P})$ is called \emph{relatively hyperbolic} when $\hat \Gamma(G, \mathcal{P}, S)$ is hyperbolic and satisfies bounded coset penetration \cite{farb-rel-hyp}. To equip $(G,\ca P)$ with a boundary, Bowditch \cite{bowditch} used an equivalent definition: $(G,\ca P)$ is relatively hyperbolic if there exists a \emph{fine} $\de$-hyperbolic graph $K$ with a $G$-action so that there are finitely many orbits of edges and $\ca P$ is the set of infinite vertex stabilizers. A graph is {\it fine} if each edge is in finitely many cycles of length $n$, for each $n$. Then the Bowditch boundary is defined as $\pa_B(G,\ca P):=\pa K\cup V_\infty(K)$, where $V_\infty(K)\sbs V(K)$ is the set of vertices of $K$ with infinite valence. If the $G$-action on $K$ is geometric, then $\ca P=\vn$ and this recovers the Gromov boundary of a hyperbolic group $\pa G=\pa K$. 

An alternate definition of relatively hyperbolic is that $G$ acts geometrically finitely on a proper geodesic metric space \cite[Definition 1]{bowditch}. In particular, this implies that for each $P\in\ca P$, the action of $P$ on $\pa_B(G,\ca P)\setminus\{*_P\}$ is properly discontinuous and cocompact. For the many equivalent notions of relative hyperbolicity, see \cite{hruskadef}. 

If $(G,\ca P)$ is relatively hyperbolic, then the coned-off Cayley graph $\hat\Ga$ is a fine hyperbolic graph \cite{Dahmanithesis}. In this case $V_\infty(\hat\Ga)\simeq\ca P$, so we can describe $\pa_B(G,\ca P)$ as
\begin{equation}\label{eqn:bowditch-boundary}\pa_B(G,\ca P)=\pa\hat\Ga\cup\big(\cup_{P\in\ca P} \{*_{P}\}\big).\end{equation}

{\it Topology on the Bowditch boundary.} For a finite subset $A\sbs V(\hat\Ga)$ and $v\in \pa\hat\Ga\cup V(\hat\Ga)$, let $M(v,A)$ denote the collection of points $w$ in $\pa_B(G,\ca P)$ so that there exists a geodesic from $v$ to $w$ that avoids $A$.  This forms a basis for the topology on $\partial_B(G, \mathcal{P})$, see \cite[Section 8]{bowditch}.  In particular, a subset $U\sbs\pa_B(G,\ca P)$ is open if for each $v\in U$, there exists a finite set $A\sbs V(\hat\Ga)$ so that $M(v,A)\sbs U$. A different basis for the topology is the collection of the sets $M_{(\la,c)}(v,A)$ of points connected to $v$ by a $(\la,c)$ quasi-geodesic that avoids $A$ (see \cite{bowditch} and \cite[\S3]{tran}).

\subsection{$\mathcal{Z}$-structures on groups} 

Before we discuss the Dahmani boundary it will be useful to have the notion of a \emph{$\ca Z$-structure} on a group \cite{bestvina_localhomology}. This concept generalizes both (i) a $\CAT(0)$-metric space with its visual boundary, and (ii) the Rips complex of a hyperbolic group with its Gromov boundary \cite{bestvina-mess}.  See \cite{ancel-guilbault} for more about $\ca Z$-structures. 

\begin{definition}[\cite{bestvina_localhomology}] A \emph{$\mathcal{Z}$-structure} on a torsion-free group $\Gamma$ is a pair of spaces $(\bar X, Z)$ such that 
\begin{enumerate} 
\item The space $\bar X$ is a Euclidean retract, i.e.\ $\bar X$ is compact, metrizable, finite dimensional, contractible, and locally contractible.
\item The subspace $Z\sbs\bar X$ is a $\ca Z$-set, i.e.\ for all $\epsilon$, there exists a map $f_\epsilon: \bar X \rightarrow \bar X \setminus Z$ that is $\epsilon$ close to the identity.
\item The space $\bar X \setminus Z$ admits a proper, cocompact $\Gamma$ action.
\item For any compact $K$ in $\bar X \setminus Z$, and any open cover $\mathcal{U}$ of $\bar{X}$, each translate $gK$ is contained in some $U_g \in \mathcal{U}$ for all but finitely many $g \in \Gamma$. 
\end{enumerate} 
\end{definition}

If $(\bar X,Z)$ is a $\ca Z$-structure on $\Ga$, then the space $Z$ is called a \emph{$\ca Z$-boundary} of $\Ga$. In general, a $\ca Z$-boundary is not unique; however,  the following theorem gives a uniqueness result for the $\ca Z$-boundary of a $\PD(n)$ group when $n\le3$. 

\begin{theorem} [\cite{bestvina-mess}]\label{thm:bestvina} Let $G$ be a torsion-free group that admits a $\mathcal{Z}$-structure $(\bar X, Z)$.  Then $G$ is a $\PD(2)$ or a $\PD(3)$ group, respectively, exactly when $Z \simeq S^1$, or $Z \simeq S^2$, respectively.  \end{theorem} 

Theorem \ref{thm:bestvina} follows directly from the proof of \cite[Cor.\ 1.3]{bestvina-mess}, together with the fact that a homology manifold that is a homology $k$-sphere is homeomorphic to $S^k$ when $k\le 2$ \cite[Rmk.\ 2.9]{bestvina_localhomology}. See also \cite[Thm 2.8]{bestvina_localhomology} for a generalization. 

\subsection{The Dahmani boundary and its topology} \label{sec:dahmani-boundary}

Fix a relatively hyperbolic group $(G,\ca P)$. Assume that each $P\in\ca P$ admits a $\ca Z$-boundary $\pa P$. As a set, the Dahmani boundary is 
\begin{equation}\label{eqn:dahmani-boundary}\pa_D(G,\ca P)=\pa\hat\Ga \cup\big(\cup_{P\in\ca P}\pa P\big).\end{equation}

If $P$ acts on $\pa P$ for each $P\in\ca P$ and if $\pa P=\pa P'$ whenever $P$ and $P'$ are conjugate, then $G$ naturally acts on $\cup_{P\in\ca P}\pa P$, and so $G$ acts on $\pa_D(G,\ca P)$. 

There is a natural map $c:\pa_D(G,\ca P)\ra\pa_B(G,\ca P)$ that is the identity on $\pa\hat\Ga$ and sends $\pa P$ to $*_{P}$. This map is studied more in \S\ref{sec:compare} and will be important in \S\ref{S:proofof2}. 

{\it Topology on the Dahmani boundary.} The topology on $\pa_D(G,\ca P)$ has a basis consisting of two types of open sets (\ref{eqn:open1}) and (\ref{eqn:open2}) below. The first type is a neighborhood basis $\{U_x'\}$ of points $x$ in $\pa \hat\Ga $.  For $x\in\pa \hat\Ga$, and for an open set $U_x\sbs\pa_B(G,\ca P)$ containing $x$, define $U_x'\sbs\pa_D(G,\ca P)$ by 
\begin{equation} \label{eqn:open1} 
U'_x =  (U_x \cap \pa \hat\Ga) \cup\big(\cup_{*_P\in U_x}\pa P\big).
\end{equation} 
The second type is a neighborhood basis about points $x \in \pa P$.  To describe it we first introduce some terminology. 

\begin{definition} 
For $P\in\ca P$ and a vertex $v\in\Ga(P)\sbs\Ga(G)$, the \emph{shadow of $v$ with respect to $P$}, denoted $\Sh(v,P)$, is the set of endpoints in $\pa\hat\Ga\cup\hat\Ga$ of (non-backtracking) geodesic arcs and rays beginning at $v$ that immediately leave $\Ga(P)$ (and do not pass through $*_{P}$). 

Furthermore, we define $\Sh_B(v,P)$ as the intersection of $\Sh(v,P)$ with $\pa_B(G,\ca P)\sbs\pa\hat\Ga\cup\hat\Ga$, and we define $\Sh_D(v,P)\sbs\pa_D(G,\ca P)$ as the preimage of $\Sh_B(v,P)$ under $c$. Note that by definition, $\Sh_B(v,P)\sbs\pa_B(G,\ca P)\sm\{*_{P}\}$. 
\end{definition} 

\begin{observation}  \label{obs:covers} For each $P\in\ca P$, 
\[\bigcup_{v \in \Ga(P)}  Sh_B(v,P) = \pa_B(G,\ca P)\sm\{*_{P}\}\>\>\text{ and so }\>\>\bigcup_{v \in \Ga(P)} Sh_D(v,P) = \pa_D(G) \setminus \pa P.\]
\end{observation} 

We now define a neighborhood basis $\{U_x'\}$ for $x\in \pa P$. For $x\in \pa P$ and a neighborhood $U_x$ of $x$ in $P\cup\pa P$, define $U_x'\sbs\pa_D(G,\ca P)$ by 
\begin{equation} \label{eqn:open2} 
U_x'= (U_x \cap \partial P) \cup\big(\cup_{v\in U_x}\Sh_D(v,P)\big) 
\end{equation} 

We recap the above discussion. 

\begin{definition} \cite[Defn 3.3]{dahmani-boundary} \label{def:Dahmani} Let $(G, \ca P) $ be a relatively hyperbolic group. Assuming each $P\in\ca P$ admits a boundary the {\it Dahmani boundary}, $\pa_D(G, \ca P)$ is the set (\ref{eqn:dahmani-boundary}) 
with topology generated by open sets of the form (\ref{eqn:open1}) and (\ref{eqn:open2}). \end{definition} 

Dahmani \cite[Thm 3.1]{dahmani-boundary} proves that $\pa_D(G,\ca P)$ is compact and metrizable. 

{\it Remark.} There is a slight difference between our definition of the topology on $\pa_D(G,\ca P)$ and the definition in \cite{dahmani-boundary}. Instead of using endpoints of \emph{geodesics} (as in our definition of $\Sh(v,P)$), Dahmani uses endpoints of quasi-geodesics that are geodesics outside of a compact set. However, these give the same topology. One way to see this is to note that $\Sh_B(v,P)$ has the form $M(v,A)$ (c.f.\ \S\ref{sec:bowditch-boundary}) where $A$ is the finite set of vertices in $\Ga(P)\cup\{*_P\}$ that are adjacent to $v$.  (Note that the distance between any two vertices in $P$ is 1 in $\hat\Gamma$.) Bowditch \cite[\S8]{bowditch} proves that this gives a basis for the topology on $\pa_B(G, \mathcal{P})$. Furthermore, Bowditch shows that this is equivalent to the topology on $\partial_B(G, \mathcal{P})$ defined using $M_{(\lambda, c)}(v, A)$, defined above. It follows that the topology we defined is equivalent to Dahmani's definition. 
\subsection{Comparing the Bowditch and Dahmani boundaries}\label{sec:compare}

Consider the {\it collapsing map} 
\begin{equation}\label{eqn:collapse-map}c:\pa_D(G,\ca P)\ra\pa_B(G,\ca P)\end{equation} that sends each peripheral boundary $\pa P$ to the corresponding point $*_P$ and is the identity on $\pa\hat \Gamma$. 

\begin{proposition}\label{prop:quotient}
Let $(G, \mathcal{P})$ be a relatively hyperbolic group. Assume that each $P\in\ca P$ admits a boundary $\pa P$. 
\begin{enumerate}
\item[(i)] For $P\in\ca P$, the inclusion $\pa P\hra\pa_D(G,\ca P)$ is an embedding. 
\item[(ii)] The subset $\bigcup_{P\in\ca P}\pa P\sbs\pa_D(G,\ca P)$ is dense, and $\{\partial P: P\in\ca P\}$ is a null family (i.e.\ for each $r>0$ there are finitely many $P\in\ca P$ with diameter greater than $r$). 
\item[(iii)] The collapsing map $c$ is continuous and $c\rest{}{\pa\hat\Ga}$ is an embedding (i.e.\ a homeomorphism onto its image).
\end{enumerate}
\end{proposition}

\begin{proof}
Both (i) and (ii) follow from the definition of the topology on $\pa_D(G,\ca P)$. The subspace topology on $\pa P\sbs\pa_D(G,\ca P)$ agrees with the standard topology on $\pa P$ by definition of the open sets (\ref{eqn:open2}). Also, $\bigcup_{P\in\ca P}\pa P$ is dense because each of the open sets (\ref{eqn:open1}) and (\ref{eqn:open2}) generating the topology on $\pa_D(G,\ca P)$ contain points of some peripheral boundary.
Finally, $\{\pa P: P\in\ca P\}$ is a null family. This is because for $r>0$, we can cover $\pa_D(G,\ca P)$ by open sets $V_1,\ldots,V_k$ of the form (\ref{eqn:open1}) or (\ref{eqn:open2}), each with diameter at most $r$, by compactness.  Note that by definition for each $V_i$ there is at most one peripheral circle that intersects $V_i$ nontrivially but is not contained in $V_i$. It follows that there are at most $k$ peripheral circles with diameter $ \geq r$.

Next we prove (iii). To show that $c$ is continuous, we fix an open set $U\sbs\pa_B(G,\ca P)$ and show $c^{-1}(U)$ is open. By definition of the topology, we can write $U$ as $U=\cup_{x\in U}M(x,A_x)$. Since each $M(x,A_x)$ is an open set containing $x$, the preimage $c^{-1}\big(M(x,A_x)\big)$ is of the form (\ref{eqn:open1}) and hence is open. Thus $c^{-1}(U)=\cup_{x\in U} \>c^{-1}\big(M(x,A_x)\big)$ is open, which implies that $c$ is continuous. To see that $c\rest{}{\pa\hat\Ga }$ is an embedding, it suffices to show that $c\rest{}{\pa\hat\Ga}$ is a closed map. This follows from the fact that $c$ is a closed map, which is true for any continuous map between compact metric spaces.
\end{proof}

\section{$\pa_D(G,\ca P)$ when $\pa_B(G,\ca P)=S^2$ (Proof of Theorem \ref{thm:sierpinski})} \label{S:proofof2}

Throughout this section we assume $\pa_B(G,\ca P)\simeq S^2$. Our goal is to show that this implies that $\pa_D(G,\ca P)$ is a Sierpinski carpet. Recall the outline of the proof of Theorem \ref{thm:sierpinski} given in the introduction. In the previous section we completed Step 1; in \S\ref{sec:identify} and \S\ref{sec:finish} we complete Steps 2 and 3, respectively. Before these steps, we explain why the Dahmani boundary is always defined when $\pa_B(G,\ca P)\simeq S^2$, i.e.\ why the peripheral subgroups admit boundaries. 

\subsection{Boundaries for peripheral subgroups} \label{sec:peripheral-boundary}

Fix $P\in\ca P$. To define a boundary $\pa P$ on $P$, consider the action of $P$ on 
\[\Om:=\pa_B(G,\ca P)\sm\{*_{P}\},\]
which is cocompact and properly discontinuous. Bowditch \cite[Section 2]{bowditchboundary} defines a metric $d_{\Om}$ on $\Om$ that makes the action of $P$ on $\Om$ geometric. Then $K=\ker\big[P\rightarrow\Isom(\Om)\big]$ is finite, and $P/K$ contains a finite-index subgroup $P'$ that is a closed surface group (in particular $P'$ is torsion free); this was observed in \cite[Theorem 0.3]{dahmani-parabolic}\footnote{According to the MathSciNet review of \cite{dahmani-parabolic}, the print version of this paper has an error. This has been fixed in an updated version (arxiv:0401059). The fact we are using is independent of this issue.}. It follows that $P$ acts geometrically on a model space $X$, either $\bb E^2$ or $\bb H^2$. Define $\pa P:=\pa X$ as the CAT(0) boundary.

Next we topologize $\ov\Om:=\Omega\cup\pa P$. By the classification of surfaces, there is a $P'$-equivariant homeomorphism $\Omega\ra X$. This extends to a map $\ov\Om\ra\ov X$, and we topologize $\ov\Om$ so that this map is a homeomorphism. 

The pair $(\ov X,\pa X)$ is the standard $\ca Z$-structure on $P$. It turns out that $(\ov\Om,\pa P)$ is an alternate description of this $\ca Z$-structure. Axioms 1--3 of a $\ca Z$-structure are immediate. Axiom 4 follows from Proposition \ref{prop:shadow} and Observation \ref{obs:covers} that the shawdows cover $\Om$. See also the proof of Theorem \ref{th:homeo}. Alternatively, one can use a very general ``boundary swapping" argument \cite[Theorem 1.3]{guilbault-moran} to conclude that $\ov\Om$ can be topologized so that $(\ov\Om,\pa P)$ is a $\ca Z$-structure on $P$. We will use our concrete description of the topology on $\ov\Om$ in what follows.

For later use, we choose a quasi-isometry $P\ra\Om$ by taking the orbit of a point. Specifically, choose a geodesic ray $\gamma_0$ in $\hat\Ga$ that starts at $*_{P}$, goes through the identity vertex $e\in \Ga(P)$, and ends at some point $0\in\pa\hat\Ga\sbs\Om$. (Recall that the boundary of a hyperbolic space consists of equivalence classes of geodesic rays, so here $\ga_0$ is a representative for $0\in\pa\hat\Ga$.) Then we identify $P$ with the orbit $P.0$. For $g \in P$, $g.0$ is the endpoint of the geodesic $g\ga_0$ in $\hat\Ga$ starting at $*_{P}$ and going through the vertex of $g$ in $\Ga(P)$.

\subsection{Identifying a Sierpinski carpet} \label{sec:identify}

The following lemma gives a criterion that will allow us to identify $\pa_D(G,\ca P)$ as a Sierpinski carpet. 

\begin{lemma}\label{lem:sierpinski}
Let $X$ be a compact metric space. Assume that there exists a continuous surjection $\pi:X\ra S^2$ such that
\begin{enumerate} 
\item[(i)] there exists a countable dense subset $Z=\{z_1,z_2,\ldots\}\sbs S^2$ so that the restriction of $\pi$ to $\pi^{-1}\big(S^2\sm Z\big)$ is injective, and
\item[(ii)] for each $k$, the space $X_k$ obtained from $X$ by collapsing each $C_i$ to a point for $i\neq k$ is homeomorphic to a closed disk $\mathbb D^2$. 
\end{enumerate}
Then $X$ is homeomorphic to a Sierpinski carpet. 
\end{lemma}

We remark that assumption (i) implies that $\pi\rest{}{}:\pi^{-1}(S^2\setminus Z)\ra S^2\setminus Z$ is a homeomorphism. Indeed $\pi\rest{}{}$ is a continuous bijection, and since $\pi$ is a continuous map between compact metric spaces, $\pi$ (and hence $\pi\rest{}{}$) is a closed map. Note also that (ii) implies that 

\begin{enumerate} 
\item[(iii)] for each $k$, the preimage $C_k:=\pi^{-1}(z_k)$ is an embedded circle. 
\end{enumerate}

\begin{example}\label{rmk:example} We illustrate the theorem with a non-example. Consider $X=[0,1]^2\sm \bigcup_{i=1}^\infty D_i$, where $D_i$ is a dense countable collection of open disks, with pairwise disjoint closures, that includes the collection of disks pictured in Figure \ref{fig:converging-disks}. The space $X$ is not homeomorphic to the Sierpinski carpet because the disks pictured in Figure \ref{fig:converging-disks} have diameter bounded from below, so $\{D_i\}$ is not a null family. Nevertheless, $X$ satisfies conditions (i) and (iii) above: the set $X\setminus\bigcup_{i=1}^\infty \ov{D_i}$ is homeomorphic to $S^2\setminus Z$, where $Z\subset S^2$ is countable and dense, and by collapsing each $\pa D_i$ to a point we obtain a continuous surjection $X\ra S^2$ satisfying (i) and (iii). Note however that condition (ii) from Lemma \ref{lem:sierpinski} is not satisfied. Indeed the space $X_1$ obtained from collapsing all the $D_i$ except the outer disk is not Hausdorff, and hence not homeomorphic to the closed disk. \end{example}

\begin{figure}[h!]
\labellist 
\pinlabel $x$ at 190 110
\endlabellist 
  \centering
\includegraphics[scale=.6]{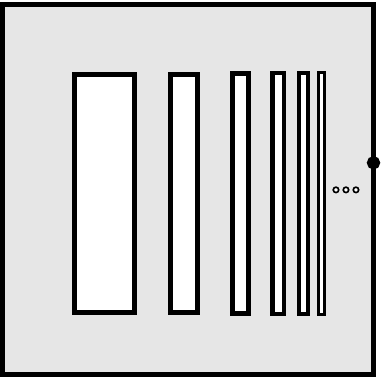}
      \caption{A collection of disjoint disks with diameter bounded from below. }
      \label{fig:converging-disks}
\end{figure}

\begin{proof}[Proof of Lemma \ref{lem:sierpinski}]
First observe that the Sierpinski carpet $X=\scr S$ satisfies the assumptions (i)--(iii) with $\pi:\scr S\ra S^2$ the map the collapses each peripheral circle to a point. Condition (ii) follows from Moore's theorem about upper semicontinuous decompositions of the plane \cite[pg 3]{daverman}. 

To prove the lemma, it suffices to show that any two compact metric spaces $X,X'$ with surjections to $S^2$ that satisfy (i)--(iii) are homeomorphic. For $k\ge0$, let $X(k)$ be the space obtained by collapsing each circle $C_i$ to a point for $i>k$ (i.e.\ we collapse all but the first $k$ circles). There are maps $p_k:X(k)\ra X(k-1)$, and $X=\lim X(k)$ is the inverse limit. Similarly, we express $X'=\lim X'(k)$. To show $X$ is homeomorphic to $X'$, we'll show that the inverse systems $\{X(k),p_k\}$ and $\{X'(k),p_k'\}$ are isomorphic. 

First we describe the topology of $X(k)$. By (ii), each $X_{k}$ is homeomorphic to $\mathbb D^2$, or equivalently $S^2\sm D$, where $D$ is an open disk. From this, it's not hard to see that $X(k)\simeq S^2\sm\bigcup_1^kD_i$, where $D_i$ are open embedded disks with disjoint closures. For example, in the case $k=2$, consider the following diagrams.
\[\begin{xy}
(0,0)*+{X(2)}="A";
(-10,-10)*+{X_{1}}="B";
(10,-10)*+{X_{2}}="C";
(0,-20)*+{X(0)}="D";
{\ar"A";"B"}?*!/_3mm/{};
{\ar "A";"C"}?*!/_3mm/{};
{\ar "B";"D"}?*!/^3mm/{};
{\ar "C";"D"}?*!/_3mm/{};
(50,0)*+{X(2)}="E";
(40,-10)*+{\bb D^2}="F";
(60,-10)*+{\bb D^2}="G";
(50,-20)*+{S^2}="H";
(25,-10)*+{\simeq}="Z";
{\ar"E";"F"}?*!/_3mm/{};
{\ar "E";"G"}?*!/_3mm/{};
{\ar "F";"H"}?*!/^3mm/{};
{\ar "G";"H"}?*!/_3mm/{};
\end{xy}
\]
Assumption (i) implies that $X(2)\ra X(0)$ is a homeomorphism away from $C_1\cup C_2$, and so $X(2)\sm(C_1\cup C_2)$ is homeomorphic to an open annulus $S^1\times(0,1)$. Furthermore, by assumption (ii), $X(2)\ra X_{i}$ is a homeomorphism in a neighborhood of $C_i$, so it follows that $X(2)$ is homeomorphic to an annulus. 

Note also that the restriction $p_{k+1}\rest{}{}:X(k+1)\setminus C_{k+1}\ra X(k)\setminus\{z_{k+1}\}$ is a homeomorphism. This follows from the definitions and assumption (i). 

We construct compatible homeomorphisms $\phi_k:X(k)\ra X'(k)$ inductively. For $k=0$, let $Z=\{z_i\}$ and $Z'=\{z_i'\}\sbs S^2$ be the countable dense subsets associated to $X,X'$. By \cite[Thm 3]{bennett} there exists a homeomorphism $\phi_0:X(0)\simeq S^2\ra S^2\simeq X'(0)$ so that $\phi_0(Z)=Z'$. Without loss of generality, we assume that $\phi_0(z_i)=z_i'$ for each $i$. For the induction step, suppose we're given a homeomorphism $\phi_{k}:X(k)\ra X'(k)$ and a commutative diagram 
\[\begin{xy}
(0,0)*+{X(k)}="A";
(20,0)*+{X'(k)}="B";
(0,-15)*+{X(0)}="C";
(20,-15)*+{X'(0)}="D";
{\ar"A";"B"}?*!/_3mm/{\phi_{k}};
{\ar "A";"C"}?*!/^3mm/{};
{\ar "B";"D"}?*!/_3mm/{};
{\ar "C";"D"}?*!/_3mm/{\phi_0};
\end{xy}\]
By the choice of $\phi_0$, it follows that $\phi_{k}(z_{k+1})=z_{k+1}'$. Then $\phi_k$ restricts to a homeomorphism $\phi_k\rest{}{}:X(k+1)\setminus C_{k+1}\ra X'(k+1)\setminus C_{k+1}'$. Since $X(k)$ is compact, $\phi_k$ is uniformly continuous, so $\phi_k\rest{}{}$ extends uniquely to a homeomorphism 
\[\phi_{k+1}:X(k+1)\ra X'(k+1)\] such that $\phi_{k+1}\circ p_{k+1}=p_{k+1}'\circ \phi_{k+1}$. This shows that the inverse systems $\{X(k),p_k\}$ and $\{X'(k),p_k'\}$ are isomorphic, so then the inverse limits $X,X'$ are homeomorphic. 
\end{proof}

\subsection{Collapsing the Dahmani boundary to a disk}\label{sec:finish}

In this section we show that $\pa_D(G,\ca P)$ and the collapse map $c:\pa_D(G,\ca P)\ra\pa_B(G,\ca P)\simeq S^2$ satisfy the assumptions of Lemma \ref{lem:sierpinski}, which allows us to conclude that $\pa_D(G,\ca P)$ is a Sierpinski carpet. The main result is as follows. 

\begin{theorem}  \label{th:homeo} Let $(G,\ca P)$ be relatively hyperbolic with $\pa_B(G,\ca P) \simeq S^2$. Fix $P\in\ca P$, and let $\pa_D(G,\ca P)_P$ be the quotient of $\pa_D(G,\ca P)$ obtained by collapsing $\pa Q$ to a point for each $Q\in\ca P\sm\{P\}$. Then $\pa_D(G,\ca P)_P$ is $P$-equivariantly homeomorphic to the disk $\overline{\Omega}$ (c.f.\ \S\ref{sec:peripheral-boundary}).
\end{theorem} 

{\it Remark.} An analogous theorem to Theorem \ref{th:homeo} holds more generally for relatively hyperbolic groups with $\pa_B(G,\ca P)\simeq S^n$ whose peripheral subgroups have $\ca Z$-boundaries (so $\Omega$ has a natural $\ca Z$-set compactification) with a similar proof.

The proof of Theorem \ref{th:homeo} will rely on the Proposition \ref{prop:shadow} below, which is a general fact about the shadow of points in the Bowditch boundary. The proof of Proposition \ref{prop:shadow} is technical, so we postpone it to the end of the section. 

\begin{proposition}\label{prop:shadow}
Let $(G, \mathcal{P})$ be a relatively hyperbolic group such that each $P\in\ca P$ admits a boundary $\pa P$. For each $P\in\ca P$, and $v\in\Ga(P)\sbs\hat\Ga$, the shadow $\Sh_B(v,P)\sbs\Om$ is bounded in the Bowditch metric on $\Om$. 
\end{proposition}

We note that since $P$ acts isometrically on $\Om$ and $\Sh_B(g\cdot v,P)=g\cdot\Sh_B(v,P)$, in fact $\Sh_B(v,P)$ is bounded uniformly for $v\in\Ga(P)$. 

\begin{proof}[Proof of Theorem \ref{th:homeo}]
There is a homeomorphism 
\[H:\pa_D(G,\ca P)_P\sm(\pa P)\ra\pa_B(G,\ca P)\sm\{*_P\},\]
since, by definition, the domain and codomain are equal as sets, and the identity map is a homeomorphism by Proposition \ref{prop:quotient}. Set $\Om:=\pa_B(G,\ca P)\sm\{*_{P}\}$ and $\ov{\Om}:=\Om\cup\pa P\simeq\bb D^2$ as in \S\ref{sec:peripheral-boundary}. Then $H$ extends (via the identity map $\pa P\ra \pa P$) to a bijection $\ov H:\pa_D(G,\ca P)_P\ra \ov{\Om}$, which is  equivariant. To prove the theorem, we need only show that $\ov H$ is a homeomorphism. 

Since $\pa_D(G,\ca P)_P$ is compact and $\ov\Om$ is Hausdorff, it suffices to show that $H$ is continuous; furthermore, since $H$ is a homeomorphism, we only need to show continuity of $\ov H$ at each $\xi\in\pa P$. Fixing $\xi\in\pa P$, it suffices to show that for every neighborhood $U$ of $\xi$ in $\ov\Om$, there exists a neighborhood $W$ of $\xi\in\pa_D(G,\ca P)$ so that $\ov H(W)\sbs U$. 

Since $Sh_B(v,P)\sbs\Om$ is bounded for each $v\in\Ga(P)$ (by Proposition \ref{prop:shadow}), there is a neighborhood $\xi\in V \subset P\cup\pa P$ such that if $v \in V$, then $\Sh_B(v,P)\sbs U$. Indeed let $\hat V$ consist of the vertices $v_g$ in $P$ such that the endpoint of $g.\gamma_0$ is in $U$, and far enough from the frontier of $U$ such that the shadow $\Sh_B(v_g,P)$ fits in $U$. Then $V$ is the interior of the closure of $\hat V$ in $P \cup \pa P$. Now the set $W_0=(V\cap \pa P)\cup \big(\cup_{v\in V}\Sh_D(v,P)\big)$ is open in $\pa_D(G,\ca P)$ (c.f.\ (\ref{eqn:open2})), and it is saturated with respect to $f_P$, so $W:=f_P(W_0)$ is open in $\pa_D(G,\ca P)_P$ and $\ov H(W)\sbs U$. This completes the proof. 
\end{proof} 

In summary, we have proved Theorem \ref{thm:sierpinski}, modulo a proof of Proposition \ref{prop:shadow}. To see this, suppose that $\partial_B(G,\ca P)\simeq S^2$. Then there is a surjection $\partial_D(G,\ca P)\ra S^2$ which satisfies the assumptions of Lemma \ref{lem:sierpinski} by Proposition \ref{prop:quotient} and Theorem \ref{th:homeo}. Then by Lemma \ref{lem:sierpinski}, $\partial_D(G,\ca P)$ is homeomorphic to the Sierpinski carpet $\scr S$. 

We remark that our understanding of the shadows allows us to prove that the complement of a point in the Bowdtich boundary admits an equivariant $Z$-set compactification (when the peripheral groups admit an equivariant $Z$-set compactification.)

\subsection{Shadows in the Bowditch boundary (Proof of Proposition \ref{prop:shadow}) } Before we begin the proof we need some additional notions and notations from \cite{bowditch}. 
When $(G,\ca P)$ is relatively hyperbolic group pair, there exists a proper hyperbolic metric space $X$ on which $G$ acts geometrically finitely. There are many models for such a space $X$, e.g.\ \cite[\S3]{bowditch} or \cite{groves-manning}, and the existence of such an $X$ is one definition of a relatively hyperbolic group pair. The main fact we will need is that the nerve of a system of horoballs in $X$ is quasi-isometric to $\hat\Ga$.

 From $X$ one can obtain a fine hyperbolic graph $K=K(X)$ by considering the nerve of an appropriate collection of horoballs $\{H(P)\}_{P\in\ca P}$ in $X$ \cite[\S7]{bowditch}. The graph $K$ has vertex set $V(K)\simeq \ca P$.
\begin{lemma}
The graph $K$ is quasi-isometric to $\hat\Ga$.
\end{lemma}
\begin{proof}
First we claim that $\hat\Ga$ is quasi-isometric to the graph $\Lam$ that has vertex set $\{*_P:P\in \ca P\}$ and an edge between $*_P$ and $*_{P'}$ if there exists an arc (i.e.\ a path with distinct vertices) between them in $\hat\Ga$ of length at most 2 such that the intermediate vertices are in $\Ga(G)\sbs\hat\Ga$. The definition of $\Lam$ is a special instance of the ``$K(A,n)$" construction in \cite[\S2]{bowditch}. To define a quasi-isometry $\Lambda\ra\hat\Ga$, note that both $\Lam$ and $\hat \Ga$ are quasi-isometric to the subset $\{*_P: P\in\ca P\}$ in each with the associated metrics, since every vertex of $\hat \Ga$ is within distance $1/2$ of some $*_P$. Then by composing, there is a map 
\begin{equation}\label{eqn:quasi-isom}\phi:\Lam\ra\hat\Ga\end{equation} that is the identity on $\{*_P: P\in\ca P\}$.  This is a quasi-isometry because 
$d_{\Lam}(*_{P_1},*_{P_2})\le d_{\hat\Ga}(*_{P_1},*_{P_2})\le 2 d_\Lam(*_{P_1},*_{P_2})$.  Notice that for any edge in $\hat \Ga$, it either meets an element of $V_\infty$, goes between two vertices at distance 1/2 from the same element of $V_\infty$, or goes between two vertices which are at distance $1/2$ from two different elements of $V_\infty$.

For any $X$ on which $(G, \mathcal{P})$ acts geometrically finitely, $\Lam$ and $K=K(X)$ are quasi-isometric because both are connected graphs with vertex set $\ca P$ and with a cocompact $G$ action; c.f.\ \cite[Lem.\ 4.2]{bowditch}. \end{proof}

It will be useful to choose a quasi-isometry 
\begin{equation}\label{eqn:qi}\pi: \hat\Ga\ra K.\end{equation} For this, it suffices to choose a coarse inverse $\psi:\hat\Ga\ra\Lam$ to the map $\phi$ in (\ref{eqn:quasi-isom}) (then we can compose with any quasi-isometry $\Lam\ra K$ that is the identity on vertices). To define $\psi$, we choose for each $v\in\Ga(G)$ an element $P_v\in\ca P$ so that $v$ is adjacent to $*_{P_v}$. If we fix $P\in\ca P$, then we can define $P_v$ as the unique subgroup (with $*_{P_v}$ adjacent to $v$) that's conjugate to $P$. Then $\psi$ is equivariant.

There is a homeomorphism  $\pa_B(G,\ca P)\ra \pa X$ \cite[\S 9]{bowditch}. Furthermore, if we label the parabolic fixed points $\Pi$ in $\pa X$ by the peripheral group $P \in \ca P $ which fixes it, then the homeomorphism from $\pa_B(G,\ca P) =\pa\hat\Ga\cup V_\infty(\hat \Ga)$ to $\partial X$ is the identity on $V_\infty(\hat\Ga)$.  Since the fixed points of the conjugates of any peripheral subgroup are dense in $\partial X$, it follows that once we fix the image of some  $*_P$ (that is, label one of the peripheral fixed points of $\partial X$) there is exactly one equivariant homeomorphism between $\pa_B(G,\ca P)$ and $\partial X$.  This allows us to canonically identify $\Omega = \partial X \setminus \{*_P\}$ with $\pa_B(G,\ca P) \setminus \{*_P\}$. 

Bowditch \cite[Section 2]{bowditchboundary} puts a metric $d_\Om$ on $\Omega$ that makes the $P$ action geometric. If two points $x,y \in \Omega$ are close in this metric, the center $z\in X$ of the ideal triangle in $X$ with vertices $x,y$ and $ *_P$ is ``close to" $\Omega$, which means that there is a horofunction $h:X\ra\R$ about $*_P$ with $h(z)\ll0$.

\begin{proof}[Proof of Proposition \ref{prop:shadow}] 
Recall from \S\ref{sec:peripheral-boundary} that we've chosen $P\hra\Om$ as the $P$-orbit of the endpoint $0\in\Om$ of a given geodesic ray $\ga_0$. We take the space $X$ with horoballs/horospheres $H(P), S(P)$, and the fine hyperbolic graph $K=K(X)$ as discussed in the preceding paragraphs. 

{\bf Step 1} (From geodesics in $\hat\Ga$ to geodesics in $X$). Suppose, for a contradiction, that the shadow of $e\in\Ga(P)$ is unbounded. Then there exist geodesics $\ga_n$ in $\hat\Ga$ from $*_P$ through $e\in\Ga(P)$ with endpoints $\xi_n\in\Om\sbs\pa_B(G,\ca P)$ such that $d_\Om(0,\xi_n)\ra\infty$. 

The image $\pi(\ga_n)$ under the quasi-isometry $\pi:\hat\Ga\ra K$ in (\ref{eqn:qi}) is a quasi-geodesic. Each $\pi(\ga_n)$ can be described as a sequence of horoballs $H(P_{n,1}), H(P_{n,2}),\ld$ in $X$, where $P_{n,1}=P$ and adjacent horoballs in this sequence are distinct.

{\it Claim.} After passing to a subsequence we can assume $H(P_{n,2})=H(P_2)$ is constant. 

{\it Proof of Claim.} 
We show there are only finitely many possibilities for the first vertex of $\pi(\ga_n)$ that differs from $*_P$. Recall that $\pi$ sends a vertex $v\in\Ga(G)$ to one of the adjacent cone vertices $*_{P_v}\in\hat \Gamma$ (such $P_v$ is conjugate to $P$), and is the identity on the cone vertices. Enumerate the vertices along the path $\ga_n$ as $(v_{n,1},v_{n,2},\ld)$. By assumption $v_{n,1}=*_P$ and $v_{n,2}=e$. By definition $\pi(v_{n,1})=\pi(v_{n,2})=*_P$. 

The first vertex of $\pi(\ga_n)$ that differs from $*_P$ will be $\pi(v_{n,3})$. There are two possibilities: either (a) $v_{n,3}$ is a cone point $*_{P_2}$ or (b) $v_{n,3}$ is a vertex of the Cayley graph $\Ga(G)$. In case (a), $\pi(v_{n,3})=*_{P_2}$, and since $*_{P_2}$ is adjacent to $e$, there are finitely many such choices.  In case (b), $\pi(v_{n,3})$ is the cone point adjacent to $v_{n,3}$ whose stabilizer is conjugate to $P$. Since $v_{n,3}$ is adjacent to $e$ and there are finitely many such vertices in $\Ga(G)$ (as $G$ is finitely generated), this shows there are finitely many possibilities for $\pi(v_{n,3})$. \qed

From $\pi(\ga_n)$, we can construct a quasi-geodesic in $X$ as follows. Let $H(P_{n,1}), H(P_{n,2}),\ld$ be the sequence of horoballs along $\pi(\ga_n)$ as defined above. For $i\ge1$, choose a geodesic arc $\al_{n,i}$ between $H(P_{n,i})$ and $H(P_{n,i+1})$ that has endpoints on the horospheres $S(P_{n,i})$ and $S(P_{n,i+1})$. Then choose a geodesic arc $\be_{n,i}$ between the endpoint of $\al_{n,i}$ and the starting point of $\al_{n,i+1}$. The concatenation $\al_{n,1}*\be_{n,1}*\al_{n,2}*\be_{n,2}*\cd$ is a quasi-geodesic with constants depending only on the quasi-geodesic constants for $\ga_n$ \cite[Lem.\ 7.3,7.6]{bowditch}. Since the $\ga_n$ have uniform constants, the quasi-geodesic $\al_{n,1}*\be_{n,1}*\cd$ is a bounded distance (with bound uniform in $n$) from a geodesic $\ga_n'$ in $X$. If $\ga_n$ represents a point $\xi_n\in\pa\hat\Ga$, then $\ga_n'$ represents the same point on $\pa X$, with respect to the natural homeomorphism $\pa_B(G,\ca P)\ra \pa X$ that takes $*_P$ to itself. 

Since the quasi-geodesics $\pi(\ga_n)$ all have the same first three vertices, there is a bounded subset of the horosphere $S(P)$ that contains $\ga_n'\cap S(P)$ for each $n$. This is because the quasi-geodesic in $X$ corresponding to $\pi(\ga_n)$, described above,  contains a geodesic segment connecting the horoballs $H(P)$ and $H(P_2)$, and any two geodesics between a pair of horoballs lie within a bounded distance from one another, c.f.\ \cite[\S9]{bowditch}.

{\bf Step 2} (Centers of ideal triangles). Let $(\xi_n)$ be the sequence of endpoints of the $\ga_n'$ in $\Omega$. Since $d_\Om(0,\xi_n)\ra\infty$ by assumption, and the $P$-action on $\Om$ is cocompact, we can chose $p_n\in P$ (with distance from $e$ in $\Ga(P)$ going to infinity) so that $d_\Om(0,p_n(\xi_n))$ is bounded. Then by passing to a subsequence we can assume that $p_n(\xi_n)$ converge in $\Om$, and in particular form a Cauchy sequence. By choosing $N$ sufficiently large, we can ensure that if $n,m>N$, then $d_\Omega\big(p_n(\xi_n),p_m(\xi_m)\big)$ is small enough to ensure that if $z_{n,m}\in X$ is the center of the ideal triangle formed by the triple $*_P$, $p_n(\xi_n)$, $p_m(\xi_m)$, then $z_{n,m}$ is disjoint from $H(P)$. 

For each $n,m>N$, we define two quasi-geodesics $\eta_{n,m}^n$ and $\eta_{n,m}^m$ between $*_P$ and $z_{n,m}$. Each is a union of two geodesic segments: for $i=n,m$, the quasi-geodesic $\eta_{n,m}^i$ follows $p_i\ga_i'$ until it nears $z_{n,m}$ and then follows a geodesic to $z_{n,m}$. Note that $\eta_{n,m}^i$ is a $(1,2c^i_{n,m})$-quasi-geodesic, where $c_{n,m}^i$ is the distance from $p_i\ga_i'$ to $z_{n,m}$. The constant $c_{n,m}^i$ is bounded in terms of the hyperbolicity constant, so the collection of quasi-geodesics $\eta_{n,m}^n$ and $\eta_{n,m}^m$ for all $n,m>N$ are all $(1,c)$-quasi-geodesics for some $c$. 

Since $z_{n,m}\notin H(P)$, the quasi-geodesics $\eta_{n,m}^i$ and $\eta_{n,m}^i$ exit $H(P)$ for $i=n,m$. Furthermore, the distance between the sets $\eta_{n,m}^n\cap S(P)$ and $\eta_{n,m}^m\cap S(P)$ is roughly comparable to the distance between $p_n$ and $p_m$ in $\Ga(P)$. This is because $\ga_n'$ and $\ga_m'$ intersect $S(P)$ in a bounded region, the intersection of $\eta_{n,m}^i$ with $S(P)$ is within the $p_i$ translate of this bounded region, and the $P$ action on the horosphere defines a quasi-isometry between the word metric on $\Ga(P)$ and the horospherical metric on $S(P)$, c.f.\ \cite[\S1]{bowditchboundary}.

{\bf Step 3} (Intrinsic and extrinsic distance in the horosphere $S(P)$). In this step we'll fix $k,\ell>N$ and consider the quasi-geodesics $\eta_{k,\ell}^k$ and $\eta_{k,\ell}^\ell$. On the one hand, $\eta_{k,\ell}^k$ and $\eta_{k,\ell}^\ell$ are a bounded distance from one another, so must exit the horoball at a bounded distance. On the other hand, the distance between $\eta_{k,\ell}^k\cap S(P)$ and $\eta_{k,\ell}^\ell\cap S(P)$ is comparable to the distance between $p_k$ and $p_\ell$ in $\Ga(P)$, which we can make as large as we want by choosing $\ell\gg k$. This tension leads to a contradiction, as we now make precise.

There is a constant $R$ so that if $\ga$ is a geodesic and $\ga'$ is a $(1,c)$-quasi-geodesic with the same endpoints, then the Hausdorff distance between $\ga,\ga'$ is less than $R$. Similarly, any two $(1,c)$-quasi-geodesics $\ga',\ga''$ with the same endpoints as $\ga$ are contained in a $2R$ neighborhood of one another. It follows that at each time $t$ the distance between $\ga'(t)$ and $\ga''(t)$ is less than $R':=4R+c$. 

According to \cite[\S1]{bowditchboundary}, the distance in $(X,\rho)$ between two points in $S(P)$ is comparable to the intrinsic metric $\si$ on $S(P)$: there are constants $K,C,\om$ so that $\si(x,y) \le K\om^{\rho(x,y)}+C$. 
Since $(S(P),\si)$ and $\Ga(P)$ are quasi-isometric, it follows that we can find $D>0$ so that if $p,q$ have distance at least $\om^D$ in $\Ga(P)$, and $x\in S(P)$, then $\rho(px,qx)> R'$. 

Choose $k>N$ and $\ell\gg k$ so that the distance between $p_k$ and $p_\ell$ in $\Ga(P)$ is greater than $\om^{D}$ (this is possible because the sequence $p_n$ is unbounded in $\Ga(P)$). Consider the $(1,c)$-quasi-geodesics $\eta_{k,\ell}^k$ and $\eta_{k,\ell}^\ell$ between $*_P$ and $z_{k,\ell}$. On the one hand, the distance between $\eta_{k,\ell}^k\cap S(P)$ and $\eta_{k,\ell}^\ell\cap S(P)$ is less than $R'$ because $\eta_{k,\ell}^k$ and $\eta_{k,\ell}^\ell$ are $(1,c)$-quasi-geodesics with the same endpoints. On the other hand, the distance between $\eta_{k,\ell}^k\cap S(P)$ and $\eta_{k,\ell}^\ell\cap S(P)$ is greater than $R'$ because $p_k,p_\ell$ have distance greater than $\om^D$ in $\Ga(P$). This contradiction implies that the shadow of a point is bounded. 
\end{proof}

\section{Corollaries to Theorem \ref{thm:sierpinski}} \label{sec:cor}

\subsection{Dahmani boundary of the double (Proof of Corollary \ref{cor:sphere})}\label{sec:double}

First we recall the definition of the double $G_\de$ of $G$ along its peripheral subgroups. We use notation similar to \cite{kk}.  

\begin{definition} Let $(G, \mathcal{P})$ be a relatively hyperbolic pair, and let $P_1,\ld,P_d$ be representatives for the conjugacy classes in $\ca P$. Define a graph of groups  $D(G,\mathcal{P})$ as follows: the underlying graph has two vertices with $n$ edges connecting them. The vertices are labeled by $G$, the $i$-th edge is labeled by  $P_i$, and the edge homomorphisms are the inclusions $P_i \hookrightarrow G$. The fundamental group of the graph of groups $D(G,\mathcal{P})$ is called the \emph{double of $G$ along $\mathcal{P}$}, denoted $G_\de$. \end{definition} 

Note that if $G$ is torsion-free, so is $G_\de$. 

\begin{proof}[Proof of Corollary \ref{cor:sphere}] Assume that $(G, \mathcal{P})$ is a torsion-free relatively hyperbolic group pair with $\partial_B(G, \mathcal{P}) \simeq S^2$.  First we remark that $(G_\de,\ca P)$ is relatively hyperbolic by work of Dahmani \cite[Thm 0.1]{dahmani-combination}. Furthermore, \cite[\S2]{dahmani-combination} describes the Bowditch boundary for graphs of groups: the result is a tree of metric spaces where the edge spaces are the limit sets of the amalgamating subgroups. (Dahmani doesn't use this terminology -- see instead Swiatkowski \cite[Defn 1.B.1]{swiatkowski}.) In the case of $G_\de$ with $\pa_B(G, \mathcal{P})=S^2$,  $\pa_B(G_\de,\ca P)$ is a ``tree of 2-spheres", where each 2-sphere has a countable dense collection of points along which other 2-spheres are glued as in the figure below. 
\begin{figure}[h]
\begin{center}
\includegraphics[scale=.8]{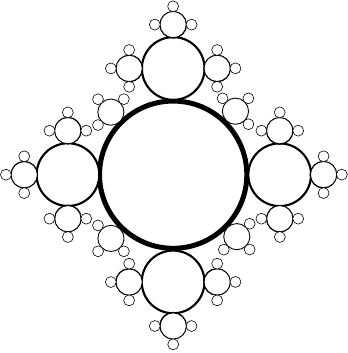}
\end{center}
\caption{The Bowditch boundary $\pa_B(G_\de,\ca P)$ is a ``tree of 2-spheres". }
\end{figure}
The Dahmani boundary inherits the structure of a tree of metric spaces from the tree structure on $\pa_B(G_\de,\ca P)$ via the collapsing map (\ref{eqn:collapse-map}) applied to $G_\de$. Each vertex space is a copy of $\pa_D(G,\ca P)$, which is a Sierpinski carpet by Theorem \ref{thm:sierpinski}. The edge spaces that meet a given vertex space are the peripheral circles $\pa P$ for $P\in\ca P$. An important part of the definition of a tree of metric spaces is that the edges spaces that meet a given vertex space must form a null family. This holds generally for the peripheral boundaries of a Dahmani boundary (Proposition \ref{prop:quotient}); it also holds in our specific case because the peripheral circles of a Sierpinski carpet are a null family \cite{cannon_sierpinski,whyburn}. It follows from \cite[Lem 1.D.2.1]{swiatkowski} that $\pa_D(G_\de,\ca P)\simeq S^2$. This completes the proof of Corollary \ref{cor:sphere}.
\end{proof}

\subsection{Duality and the Bowditch boundary (Corollary \ref{cor:duality} and its Converse)}

\begin{proof}[Proof of Corollary \ref{cor:duality}]

By a criterion of Bieri--Eckmann \cite[Cor 8.5]{bieri-eckmann-pdpairs}, to show that $(G,\ca P)$ is a PD(3) pair, it is enough to show that the double $G_\de$ is a PD(3) group and that the peripheral subgroups $P\in\ca P$ are $\PD(2)$ groups. The latter is true because the peripheral subgroups act properly and cocompactly on $\pa_B(G,\ca P)\sm\{*_P\}\simeq\R^2$, c.f.\ \cite[Thm 0.3]{dahmani-parabolic} and the assumption that our group is torsion-free. To see $G_\de$ is a PD(3) group, we use Corollary \ref{cor:sphere} to conclude $\pa_D(G_\de,\ca P)\simeq S^2$. Since $\pa_D(G_\de,\ca P)$ is a $\ca Z$-boundary for $G_\de$ \cite[Thm 0.2]{dahmani-boundary}, and $G_\de$ is torsion free, it follows that $G_\de$ is a PD(3) group by the argument of Bestvina--Mess \cite[Corollary 1.3 (b,c)]{bestvina-mess}. (See Theorem \ref{thm:bestvina} above.) 
\end{proof}

\begin{proof}[Proof of Theorem \ref{th:converse}] Let $(G, \mathcal{P})$ be a relatively hyperbolic group pair which is also a $\PD(3)$ pair.  It follows that $G$ is torsion-free and again by  \cite[Cor 8.5]{bieri-eckmann-pdpairs}, the subgroups in $\mathcal{P}$ are surface groups, and the double of $G$ along $\mathcal{P}$ is a $\PD(3)$ group.  By \cite[Thm 0.1]{dahmani-combination} $(G_\delta, \mathcal{P})$ is relatively hyperbolic, so $G_\delta$ admits a $Z$-structure with $\ca Z$-boundary $\pa_D(G_\de,\ca P)$ by Dahmani \cite{dahmani-boundary}.  It follows that $\partial_D(G_\delta, \mathcal{P})\simeq S^2$, c.f.\ Theorem \ref{thm:bestvina}.  By Proposition \ref{prop:quotient},  there is a dense collection of embedded circles in $\partial_D(G_\delta, \mathcal{P})$ such that when we form the quotient by collapsing these circles, we obtain $\pa_B(G_\de,\ca P)$. As each embedded circle in $S^2$ bounds a disk on either side, the result is a tree of 2-spheres glued along points. By \cite[Theorem 0.1]{bowditchboundary} and \cite[Theorem 9.2]{bowditchperipheral}, each of these cut points correspond to a peripheral splitting.  Furthermore, by the description of the boundary of an amalgamated product given in \cite[section 2]{dahmani-combination}, this tree of two-spheres is formed by gluing the Bowditch boundaries of the vertex groups along the limit sets of the amalgamating groups, which are the fixed points of the peripheral subgroups in this case.  Thus, the Bowditch boundary of each vertex group (relative to $\mathcal{P}$)  is $S^2$, hence $\pa_B(G, \mathcal{P})\simeq S^2$. 
\end{proof} 

\subsection{The Wall and relative Cannon conjectures (Proof of Theorem \ref{thm:wall})} \label{s:wall} 

Let $(G, \mathcal{P})$ be a  relatively hyperbolic group pair with $G$ torsion-free and $\pa_B(G, \mathcal{P})\simeq S^2$. We may assume $\mathcal{P}$ is non-empty. Choose representatives of the conjugacy classes of the peripheral subgroups $P_1,\ld, P_d$ and denote our group pair by $(G, \lb P_i \rb)$. Corollary \ref{cor:duality} implies that the double $G_\delta$ is $\PD(3)$ group. Assuming the Wall conjecture, we conclude that $G_\de=\pi_1(M)$ for some closed aspherical 3-manifold. 

Let $M'\ra M$ be the cover corresponding to $G<G_\de$. Since $G$ is finitely generated, by Scott's compact core theorem \cite{scott-compact-core}, there is a compact submanifold $N\sbs M'$ such that the inclusion induces an isomorphism $\pi_1(N)\simeq\pi_1(M')\simeq G$. Let $N_0$ be $N$ without its torus boundary components. To prove the theorem, we explain why $N_0$ admits a complete hyperbolic metric with totally geodesic boundary, and that the boundary subgroups and cusp subgroups are exactly the peripheral subgroups of $(G, \mathcal{P})$. 

{\it Claim.} (i) Any $\Z\ti\Z$ subgroup of $\pi_1(N)$ is conjugate into one of the boundary subgroups. (ii) The boundary subgroups are malnormal, i.e., if $P_i \cap {}^gP_j \neq \{1\}$ for any two boundary subgroups $P_i$ and $P_j$,  then $P_i = P_j$ and $g \in P_i $. 

To prove the claim, first note that any $\Z\ti\Z$ subgroup of a relatively hyperbolic group is contained in one of the peripheral subgroups. To see this, consider a geometrically finite action of $G$ on a hyperbolic space, and use the classification of isometries \cite[Prop.\ 4.1]{benakli-kapovich}. Now the claim follows once we explain that the boundary subgroups of $N$ and the peripheral subgroups $P_1,\ld,P_n$ are the same, up to conjugacy. (This justifies our notation in (ii).) This follows from the uniqueness of the $\PD(3)$-pair structure for pairs $(G,\{P_1,\ld,P_n\})$, where the subgroups $P_1,\ld,P_n$ do not \emph{coarsely separate} $G$ \cite[Thm 1.5]{kk_duality}. In our case $P_i<G$ does not coarsely separate because $\pa P_i\sbs\pa_D(G,\ca P)$ does not separate as the peripheral circles of a Sierpinski carpet do not separate; they are exactly the non-separating circles. Malnormality of the peripheral subgroups in torsion-free relatively hyperbolic groups is exactly \cite[Prop.\ 2.37]{osin}.  This finishes the proof of the claim.

Since every $\Z \ti \Z$ subgroup is peripheral,  $N_0$ admits a complete hyperbolic metric. To see this, observe that if $N_0$ has no higher genus boundary components, this is  Thurston's hyperbolization \cite[Theorem B]{morgan-uniformization}.  Suppose $N_0$ has higher genus boundary components. Then there are no essential annuli since this would yield a free homotopy between two curves on the boundary, impling that the group elements are conjugate. Malnormality implies that this conjugation can be done in the surface group, so the annulus is not essential.   Thus the double of $N_0$ along the higher genus boundary components is hyperbolic \cite[Theorem B]{morgan-uniformization}, and the involution of the double fixes the boundary components of $N_0$.  Since this is realized by an isometry \cite[Theorem 1.44]{Kapovichbook},  $N_0$ admits a hyperbolic metric with totally geodesic boundary components.

\bibliographystyle{alpha}
\bibliography{tpbib}

Department of Mathematics, Harvard University\\ 
\texttt{tshishikub@gmail.com}

Department of Mathematics, Tufts University\\
\texttt{genevieve.walsh@tufts.edu}

\includepdf[pages=-]{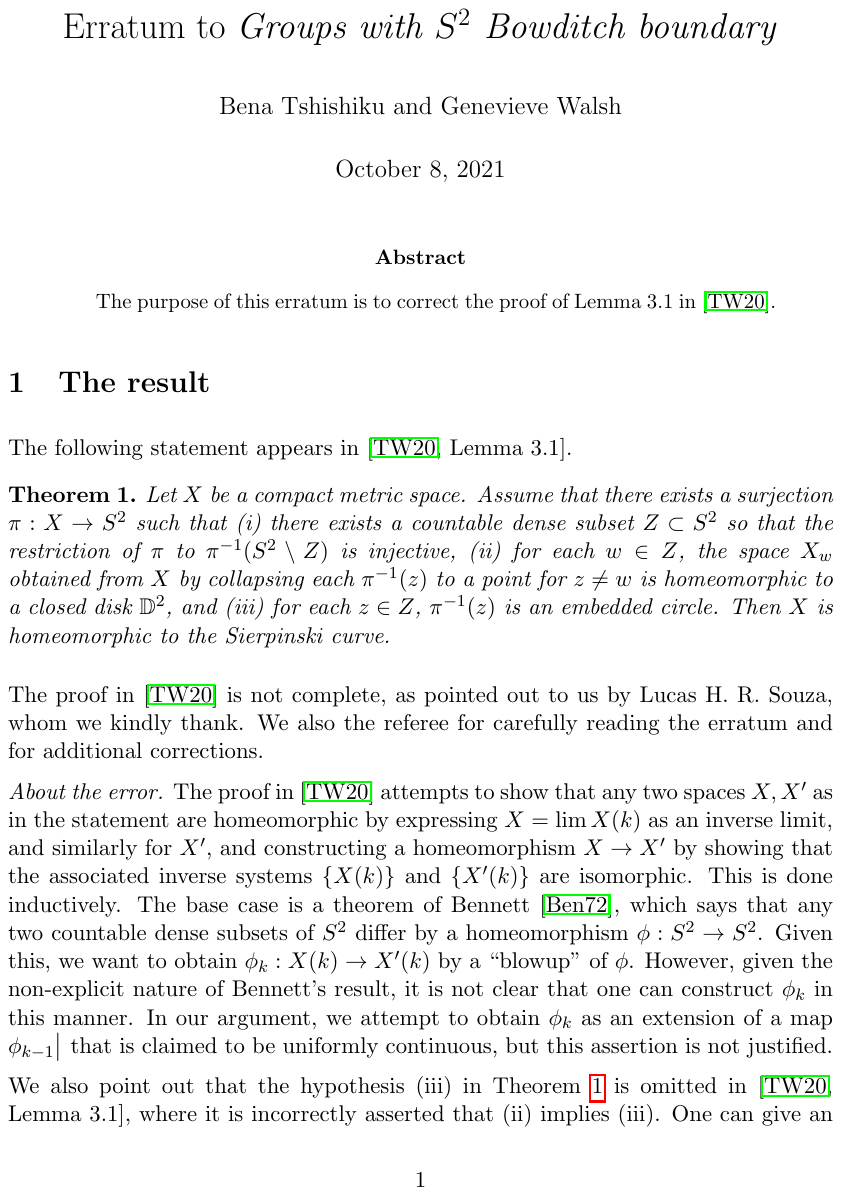}

\end{document}